\documentclass[12pt]{amsart}

\usepackage[T2A]{fontenc}
\usepackage{amsfonts}
\usepackage{amsbsy}
\usepackage{amssymb}
\usepackage[english]{babel}
\usepackage{xypic}
\usepackage{hyperref}

\renewcommand{\abovecaptionskip}{0pt}
\renewcommand{\belowcaptionskip}{6pt}
\makeatletter

\renewcommand{\@makecaption}[2]{
\vspace{\abovecaptionskip}%
\sbox{\@tempboxa}{#1 #2}%
\global\@minipagefalse \hbox to \hsize {{\scshape \hfil #1 #2\hfil}}
\vspace{\belowcaptionskip}}

\makeatother

\newcommand{\quot}{/\hspace{-0.5ex}/}

\renewcommand{\ge}{\geqslant}
\renewcommand{\le}{\leqslant}

\newcommand{\rk}{\operatorname{rk}}
\newcommand{\codim}{\operatorname{codim}}
\newcommand{\SL}{\operatorname{SL}}
\newcommand{\GL}{\operatorname{GL}}
\newcommand{\Sp}{\operatorname{Sp}}
\newcommand{\Spin}{\operatorname{Spin}}
\newcommand{\SO}{\operatorname{SO}}
\newcommand{\diag}{\operatorname{diag}}
\newcommand{\Spec}{\operatorname{Spec}}

\DeclareMathOperator{\otimesZ}{\otimes\hspace{1pt}\rule[-3pt]{0pt}{0pt}_{\mathbb{Z}}}
 \DeclareMathOperator{\Supp}{Supp}

\newtheorem{theorem}{Theorem}

\newtheorem{proposition}{Proposition}
\newtheorem{lemma}{Lemma}
\newtheorem{corollary}{Corollary}

\newtheorem*{question*}{Question}

\theoremstyle{definition}

\newtheorem{dfn}{Definition}
\newtheorem*{dfn*}{Definition}

\theoremstyle{remark}

\newtheorem{remark}{Remark}

\begin{document}

\renewcommand{\proofname}{Proof}
\renewcommand{\abstractname}{Abstract}
\renewcommand{\refname}{References}
\renewcommand{\figurename}{Figure}
\renewcommand{\tablename}{Table}

\title[Affine spherical homogeneous spaces]{Affine spherical homogeneous
spaces\\with good quotient by a maximal unipotent subgroup}

\author{Roman Avdeev}

\thanks{Partially supported by Russian Foundation for Basic Research, grant no. 09-01-00648}

\address{Chair of Higher Algebra, Department of Mechanics and Mathematics,
Moscow State University, 1, Leninskie Gory, Moscow, 119992, Russia}

\email{suselr@yandex.ru}




\begin{abstract}
For an affine spherical homogeneous space $G/H$ of a connected
semisimple algebraic group~$G$, we consider the factorization
morphism by the action on $G/H$ of a maximal unipotent subgroup
of~$G$. We prove that this morphism is equidimensional if and only
if the weight semigroup of $G/H$ satisfies some simple condition.
\end{abstract}

\maketitle

\sloppy

\section{Introduction} \label{section_intro}

Let $G$ be a connected semisimple complex algebraic group and let
$U$ be a maximal unipotent subgroup of~$G$. Consider a homogeneous
space $G/H$ with finitely generated algebra $\mathbb C[G/H] =
\mathbb C[G]^H$ of regular functions on it. In this situation, the
algebra ${}^U \mathbb C[G/H]$, which consists of regular functions
on $G/H$ that are invariant under the action of $U$ on the left, is
also finitely generated (see~\cite[Theorem~3.1]{Kh}). Therefore one
may consider the corresponding factorization morphism
$$
\pi_U \colon X = \Spec \mathbb C[G/H] \to Y = \Spec {}^U \mathbb
C[G/H].
$$

The algebra $\mathbb C[G/H]$ is a rational $G$-module with respect
to the action of $G$ on the left and decomposes into a direct sum of
finite-dimensional irreducible $G$-modules. The highest weights of
irreducible $G$-modules that occur in this decomposition form a
semigroup called the \textit{weight semigroup} of~$G/H$. We denote
this semigroup by~$\Lambda_+(G/H)$.

A subgroup $H \subset G$ (or a homogeneous space $G/H$) is said to
be \textit{spherical} if a Borel subgroup $B \subset G$ has an open
orbit in~$G/H$. For a spherical homogeneous space $G/H$, the algebra
$\mathbb C[G/H]$ is finitely generated~\cite{Kn}, therefore the
morphism $\pi_U$ is well defined. Further, it is known
(see~\cite[Theorem~1]{VK}) that for a spherical homogeneous space
$G/H$ the $G$-module $\mathbb C[G/H]$ is \textit{multiplicity free}
(the converse is also true in the case of quasi-affine $G/H$), that
is, every irreducible submodule occurs in this $G$-module with
multiplicity at most~$1$. In this situation, if we fix a highest
weight vector (with respect to~$U$) in each irreducible submodule
of\, $\mathbb C[G/H]$, then these vectors form a basis of the
algebra ${}^U \mathbb C[G/H]$ (regarded as a vector space
over~$\mathbb C$). Moreover, if we normalize these vectors in an
appropriate way, we can establish a natural isomorphism between
${}^U \mathbb C[G/H]$ and the semigroup algebra of the semigroup
$\Lambda_+(G/H)$ (see~\cite[Theorem~2]{Pop}).

Let $\omega_1, \ldots, \omega_l$ be all the fundamental weights
of~$G$. For every dominant weight $\lambda = k_1\omega_1 + \ldots +
k_l\omega_l$, $k_i \in \{0,1,2,\dots\}$, we introduce its
\textit{support} $\Supp \lambda = \{\omega_i \mid k_i > 0\}$.

A spherical homogeneous space $G/H$ is said to be \textit{excellent}
if it is quasi-affine and the semigroup $\Lambda_+(G/H)$ is
generated by dominant weights $\lambda_1, \ldots, \lambda_m$ of~$G$
satisfying $\Supp \lambda_i \cap \Supp \lambda_j = \varnothing$ for
$i \ne j$. For example, regard the (affine) spherical homogeneous
space $\SL_{2n} / \mathrm{S}(\mathrm{L}_n \times \mathrm{L}_n)$,
where $\mathrm{S}(\mathrm{L}_n \times \mathrm{L}_n)$ is the
intersection of the group $\SL_{2n}$ with the subgroup $\GL_n \times
\GL_n \subset \GL_{2n}$. Its weight semigroup is (freely) generated
by the weights $\omega_1 + \omega_{2n-1}$, $\omega_2 +
\omega_{2n-2}$, $\ldots$, $\omega_{n-1} + \omega_{n+1}$,
$2\omega_{n}$, where $\omega_i$ is the $i$th fundamental weight of
$\SL_{2n}$, hence this space is excellent. It follows from the
definition that for an excellent spherical homogeneous space $G/H$
the semigroup $\Lambda_+(G/H)$ is free. Therefore the algebra ${}^U
\mathbb C[G/H]$ is also free and the variety~$Y$, which is the
spectrum of this algebra, is just the affine space $\mathbb C^r$,
where $r$ is the rank of~$\Lambda_+(G/H)$.

A spherical homogeneous space $G/H$ is said to be \textit{almost
excellent} if it is quasi-affine\footnote{In the paper~\cite{Avd1},
the quasi-affinity condition was omitted by mistake in the
definition of an almost excellent spherical homogeneous space.} and
the convex cone $\mathbb Q_+ \Lambda_+(G/H)$ consisting of all
linear combinations of elements in $\Lambda_+(G/H)$ with nonnegative
rational coefficients is generated (as a convex cone) by elements
$\lambda_1, \ldots, \lambda_m \in \Lambda_+(G/H)$ satisfying $\Supp
\lambda_i \cap \Supp \lambda_j = \varnothing$ for $i \ne j$. As can
be easily seen (see~\cite[Corollary~1]{Avd1}), for a spherical
homogeneous space $G/H$ the property of being almost excellent is
local, that is, it depends only on the Lie algebras~$\mathfrak g$
and~$\mathfrak h$.

There is a close connection between excellent and almost excellent
spherical homogeneous spaces. First, it is obvious that every
excellent spherical homogeneous space is almost excellent. Second,
for every almost excellent spherical homogeneous space its simply
connected covering homogeneous space is excellent
(see~\cite[Theorem~3]{Avd1}).

The following theorem was in fact proved by Panyushev in~1999.

\begin{theorem} \label{excellent_ed}
Let $G/H$ be a quasi-affine spherical homogeneous space.

\textup{(a) (see \cite[Theorem~5.5]{Pan2})} If\, $G/H$ is excellent,
then the morphism $\pi_U$ is equidimensional.

\textup{(b) (see \cite[Theorem~5.1]{Pan2})} If $Y \simeq \mathbb
C^r$ for some~$r$ and $H$ contains a maximal unipotent subgroup
of~$G$, then the converse to~\textup{(a)} is also true.
\end{theorem}

\begin{remark}
When we say that a morphism is equidimensional, we have in mind that
it is surjective.
\end{remark}

\begin{remark}
In fact, Theorems 5.1 and 5.5 in~\cite{Pan2} assert more general
facts, formulated in other terms. The term `excellent spherical
homogeneous space' appeared in~2007 when Theorem~\ref{excellent_ed}
was independently reproved by E.\,B.~Vinberg and S.\,G.~Gindikin
(unpublished).
\end{remark}

In view of the above-mentioned connection between excellent and
almost excellent spherical homogeneous spaces,
Theorem~\ref{excellent_ed} implies the following similar result for
almost excellent spherical homogeneous spaces (see also
\S\,\ref{section_reformulation}).

\begin{corollary}\label{alm_exc_ed}
Let $G/H$ be a quasi-affine spherical homogeneous space. Then:

\textup{(a)} if $G/H$ is almost excellent, then the morphism $\pi_U$
is equidimensional;

\textup{(b)} if $H$ contains a maximal unipotent subgroup of~$G$,
then the converse to~\textup{(a)} is also true.
\end{corollary}

As can be seen, Theorem~\ref{excellent_ed}(b) (resp.
Corollary~\ref{alm_exc_ed}(b)) is the converse of
Theorem~\ref{excellent_ed}(a) (resp. of
Corollary~\ref{alm_exc_ed}(a)) for spherical homogeneous
spaces~$G/H$ such that $H$ contains a maximal unipotent subgroup
of~$G$ (such subgroups $H$ are said to be \textit{horospherical}).

The goal of this paper is to establish the converse of
Theorem~\ref{excellent_ed}(a) and Corollary~\ref{alm_exc_ed}(a) in
the case of \textit{affine} spherical homogeneous spaces~$G/H$, that
is, in the case where $H$ is reductive. Namely, in this paper we
prove the following theorem.

\begin{theorem} \label{ed_exc}
Let $G/H$ be an affine spherical homogeneous space. Then:

\textup{(a)} if the morphism $\pi_U$ is equidimensional, then $G/H$
is almost excellent;

\textup{(b)} if the morphism $\pi_U$ is equidimensional and $Y
\simeq \mathbb C^r$ for some~$r$, then $G/H$ is excellent.
\end{theorem}

Doubtless, it would be interesting to understand to what extent this
result can be generalized to the case of arbitrary quasi-affine
spherical homogeneous spaces.

Theorem~\ref{excellent_ed}(a), Corollary~\ref{alm_exc_ed}(a), and
Theorem~\ref{ed_exc} imply the following geometric characterization
of excellent and almost excellent affine spherical homogeneous
spaces.

\begin{corollary}
Let $G/H$ be an affine spherical homogeneous space. Then:

\textup{(a)} $G/H$ is almost excellent if and only if the morphism
$\pi_U$ is equidimensional;

\textup{(b)} $G/H$ is excellent if and only if the morphism $\pi_U$
is equidimensional and $Y \simeq \mathbb C^r$ for some~$r$.
\end{corollary}

The paper is organized as follows.
In~\S\,\ref{section_reformulation} we reformulate
Theorem~\ref{ed_exc} in a form that is more convenient to prove (see
Theorem~\ref{theorem_reformulation}).
In~\S\S\,\ref{section_classification}--\ref{section_ews} we collect
all the  and results needed for the proof of
Theorem~\ref{theorem_reformulation}. Namely, in
\S\,\ref{section_classification} we recall the classification of
affine spherical homogeneous spaces; in~\S\,\ref{section_zero_fiber}
we prove that under some restrictions on a homogeneous space $G/H$
the null fiber of the morphism $\pi_U$ is nonempty;
in~\S\,\ref{section_sym_act} we consider symmetric linear actions of
tori; in~\S\,\ref{section_ews} we recall the notion of the extended
weight semigroup of a homogeneous space. We prove
Theorem~\ref{theorem_reformulation} in
\S\S\,\ref{section_reduction}--\ref{section_str_irr}. More
precisely, in~\S\,\ref{section_reduction} we reduce the proof of
this theorem to the case of strictly irreducible spaces, which in
turn is considered in~\S\,\ref{section_str_irr}.

The author expresses his gratitude to E.\,B.~Vinberg for a
discussion about this paper and valuable comments.

\section{Some conventions and notation}\label{section_notation}

In this paper the base field is the field $\mathbb C$ of complex
numbers, all topological terms relate to the Zarisky topology, all
groups are assumed to be algebraic and their subgroups closed. The
Lie algebras of groups denoted by capital Latin letters are denoted
by the corresponding small German letters. For every group $L$ we
denote by $\mathfrak X(L)$ the character lattice of~$L$.

Throughout the paper, $G$ stands for a connected semisimple
algebraic group. We assume a Borel subgroup $B \subset G$ and a
maximal torus $T \subset B$ to be fixed. We denote by~$U$ the
maximal unipotent subgroup of $G$ contained in~$B$. We identify the
lattices $\mathfrak X(B)$ and $\mathfrak X(T)$ by restricting the
characters from~$B$ to~$T$.

The actions of $G$ on itself by left translation ($(g,x) \mapsto
gx$) and right translation ($(g,x) \mapsto xg^{-1}$) induce its
representations on the space $\mathbb C[G]$ of regular functions
on~$G$ by the formulae $(gf)(x) =f(g^{-1}x)$ and $(gf)(x) =f(xg)$,
respectively. For brevity, we refer to these actions as the action
\textit{on the left} and \textit{on the right}, respectively. For
every subgroup $L\subset G$ we denote by ${}^L\mathbb C[G]$ (resp.
by $\mathbb C[G]^L$) the algebra of functions in $\mathbb C[G]$ that
are invariant under the action of $L$ on the left (resp. on the
right).

Two homogeneous spaces $G_1/H_1$ and $G_2/H_2$ are said to be
\textit{locally isomorphic} if their simply connected covering
homogeneous spaces are isomorphic. This is equivalent to the
existence of an isomorphism $\mathfrak g_1 \to \mathfrak g_2$ taking
$\mathfrak h_1$ to~$\mathfrak h_2$.

Without loss of generality, for every simply connected homogeneous
space $G/H$ we assume that $G$ is simply connected and $H$ is
connected.

For a group $L$, the notation $L = L_1 \cdot L_2$ signifies that $L$
is an \textit{almost direct product} of subgroups $L_1, L_2 \subset
L$, that is, $L = L_1 L_2$, the subgroups $L_1, L_2$ elementwise
commute and the intersection $L_1 \cap L_2$ is finite.

Notation:

$e$ is the identity element of an arbitrary group;

$\mathbb C^\times$ is the multiplicative group of the field~$\mathbb
C$;

$V^*$ is the space of linear functions on a vector space~$V$;

$\Lambda_+(G) \subset \mathfrak X(B)$ is the semigroup of dominant
weights of $G$ with respect to~$B$;

$V(\lambda)$ is the irreducible $G$-module with highest weight
$\lambda \in \Lambda_+(G)$;

$v_\lambda \in V(\lambda)$ is a highest weight vector in
$V(\lambda)$ with respect to~$B$;

$\lambda^*$ is the highest weight of the irreducible $G$-module
$V(\lambda)^*$;

$L^0$ is the connected component of the identity of a group~$L$;

$L'$ is the derived subgroup of a group~$L$;

$Z(L)$ is the center of a group~$L$;

$N_L(K)$ is the normalizer of a subgroup $K$ in a group~$L$;

$\diag (a_1, \ldots, a_n)$ is the diagonal matrix of order~$n$ with
elements $a_1, \ldots, a_n$ on the diagonal.

\section{Reformulation of the main theorem}\label{section_reformulation}

In this section we reduce proving Theorem~\ref{ed_exc} to proving
the following theorem.

\begin{theorem}\label{theorem_reformulation}
Suppose that $G/H$ is a simply connected affine spherical
homogeneous space such that the morphism $\pi_U$ is equidimensional.
Then $G/H$ is excellent.
\end{theorem}

Let $G/H$ be a simply connected spherical homogeneous space. Every
homogeneous space that is locally isomorphic to $G/H$ has the form
$G/\widetilde H$, where $\widetilde H$ is a finite extension of~$H$,
that is, $\widetilde H^0 = H$. Put $\widetilde Y = \Spec {}^U
\mathbb C[G/\widetilde H]$ and consider the following commutative
diagram:
\begin{equation*}
\xymatrix{
G/H \ar[r]^{\pi_U} \ar[d] & Y \ar[d]\\
G/\widetilde H \ar[r]_{\widetilde\pi_U} & \widetilde Y}
\end{equation*}
In this diagram, the vertical arrows correspond to the factorization
morphisms by the finite group $\widetilde H/H$, which implies that
the condition that $\pi_U$ be equidimensional is equivalent to the
condition that~$\widetilde \pi_U$ be equidimensional. In view of
this, Theorem~\ref{theorem_reformulation} implies
Theorem~\ref{ed_exc}(a). Now let us deduce part~(b) of
Theorem~\ref{ed_exc} from part~(a). For that, we note that an almost
excellent spherical homogeneous space $G/H$ is excellent if and only
if the semigroup $\Lambda_+(G/H)$ is free. The latter is equivalent
to the condition that the algebra ${}^U \mathbb C[G/H]$ be free,
that is, $Y \simeq \mathbb C^r$ for some~$r$.

\section{Classification of affine spherical homogeneous spaces} \label{section_classification}

In this section we recall the known classification, up to a local
isomorphism, of all affine spherical homogeneous spaces or,
equivalently, the classification, up to an isomorphism, of all
simply connected affine spherical homogeneous spaces. Before we
proceed, let us recall some notions.

A direct product of spherical homogeneous spaces
$$
(G_1 / H_1) \times (G_2 / H_2) = {(G_1 \times G_2)} / (H_1 \times
H_2)
$$
is again a spherical homogeneous space. Spaces of this kind, as well
as spaces locally isomorphic to them, are said to be
\textit{reducible} and all others are said to be
\textit{irreducible}. A~spherical homogeneous space $G/H$ is said to
be \textit{strictly irreducible} if the spherical homogeneous space
$G/N_G(H)^0$ is irreducible (see~\cite[\S\,1.3.6]{Vin}).

Up to a local isomorphism, the list of all strictly irreducible
affine spherical homogeneous spaces $G/H$ is known: in the case of
simple~$G$ it was obtained in~\cite{Kr} and in the case of nonsimple
semisimple~$G$ it was obtained in~\cite{Mi} and, independently and
by another method, in~\cite{Br}. This list is collected in its
entirety in Tables~1 and~2 of the paper~\cite{Avd1}. (Although
different parts of this list can be found in many other papers
including the original papers mentioned above.) In its final shape,
the general procedure of obtaining arbitrary affine spherical
homogeneous spaces starting from strictly irreducible ones is given
in~\cite{Yak}\footnote{Originally a procedure of this kind was
suggested in~\cite{Mi}, however it proved to be wrong. In~\cite{Yak}
this error is pointed out and a correct version of the procedure is
given.} (see its description below).

We note that for all simply connected strictly irreducible affine
spherical homogeneous spaces $G/H$ the corresponding weight
semigroups are known. In the case of simple $G$ these semigroups
were computed in~\cite{Kr} and in the case of nonsimple
semisimple~$G$ they were computed in~\cite{Avd2}. (More precisely,
in~\cite{Avd2} the corresponding extended weight semigroups were
computed; these are defined in~\S\,\ref{section_ews}).

We subdivide the strictly irreducible affine spherical homogeneous
spaces into the following three types:

type~I: $\dim Z(H) = 1$, the space $G/H'$ is not spherical;

type~II: $\dim Z(H) = 1$, the space $G/H'$ is spherical;

type~III: $\dim Z(H) = 0$.

We now describe the general procedure of constructing arbitrary
simply connected affine spherical homogeneous spaces starting with
the strictly irreducible ones. Let $G_1/H_1$, $\ldots$, $G_n/H_n$ be
simply connected affine spherical homogeneous spaces. We recall that
the groups $G_1, \ldots, G_n$ are assumed to be simply connected and
the subgroups $H_1, \ldots, H_n$ are assumed to be connected.
Renumbering, if necessary, we may assume that for some $p,q$, where
$0 \le p \le q \le n$, the spaces $G_1/H_1, \ldots, G_p/H_p$ are of
type~I, the spaces $G_{p+1}/H_{p+1}, \ldots, G_q/H_q$ are of
type~II, and the spaces $G_{q+1}/H_{q+1}, \ldots, G_n/H_n$ are of
type~III. We put $G = G_1 \times \ldots \times G_n$, $\widetilde H =
H_1 \times \ldots \times H_n$. Clearly, $Z(\widetilde H) = Z(H_1)
\times \ldots \times Z(H_n)$ and $\widetilde H' = H'_1 \times \ldots
\times H'_n$. Further, for all $i = 1, \ldots, n$ we have $\rk
\mathfrak X(H_i) = \dim Z(H_i)$, whence $\mathfrak X(\widetilde H)
\simeq \mathfrak X(H_1) \oplus \ldots \oplus \mathfrak X(H_q)$ is a
lattice of rank~$q$. Let $\chi_1, \ldots, \chi_p$ denote the images
in $\mathfrak X(\widetilde H)$ of basis elements of the lattices
$\mathfrak X(H_1), \ldots, \mathfrak X(H_p)$, respectively.

Let $Z$ be a connected subgroup of $Z(\widetilde H)^0 = Z(H_1)
\times \ldots \times Z(H_q)$. Put $H = {Z \cdot (H'_1 \times \ldots
\times H'_n)} \subset \widetilde H$. Then we can consider the
character restriction map $\tau: \mathfrak X(\widetilde H) \to
\mathfrak X(H)$.

\begin{theorem}[{\cite[Theorem~3, Lemma~5]{Yak}}] \label{classification}
\textup{(a)} The space $G/H$ obtained using the above procedure is
spherical if and only if the characters $\tau(\chi_1), \ldots,
\tau(\chi_p)$ are linearly independent in~$\mathfrak X(H)$.

\textup{(b)} Every simply connected affine spherical homogeneous
space is isomorphic to one of the spaces~$G/H$ obtained using the
above procedure.
\end{theorem}

\section{The null fiber of the morphism~$\pi_U$} \label{section_zero_fiber}

The main result in this section is Proposition~\ref{non_empty}.

The following isomorphism of $(G \times G)$-modules is well known:
\begin{equation} \label{relation}
\mathbb C[G] \simeq \bigoplus \limits_{\lambda \in \Lambda_+(G)}
V(\lambda) \otimes V(\lambda^*),
\end{equation}
where on the left-hand side the group $G \times G$ acts on the left
and on the right and in each summand on the right-hand side the left
(resp. right) factor of $G \times G$ acts on the left (resp. right)
tensor factor. Under this isomorphism, an element $v \otimes w \in
V(\lambda) \otimes V(\lambda^*)$ corresponds to the function in
$\mathbb C[G]$ whose value at an element $g \in G$ is $\langle
g^{-1}v, w \rangle = \langle v, gw \rangle$, where
$\langle\,\cdot\,,\,\cdot\, \rangle$ is the natural pairing between
$V(\lambda)$ and~$V(\lambda^*) \simeq V(\lambda)^*$.

As can be easily seen, for a fixed subgroup~$H \subset G$ the
subspace ${}^U \mathbb C[G/H]_\lambda \subset {}^U \mathbb C[G/H]$,
consisting of all $T$-semi-invariant functions of weight~$\lambda$,
corresponds to the subspace $\langle v_\lambda \rangle \otimes
V(\lambda^*)^H \subset V(\lambda) \otimes V(\lambda^*)$ under
isomorphism~(\ref{relation}). In particular, this implies that for
an element $\lambda \in \Lambda_+(G)$ the condition $\lambda \in
\Lambda_+(G/H)$ is equivalent to the condition that the subspace
$V(\lambda^*)^H \subset V(\lambda^*)$, consisting of all
$H$-invariant vectors, is nontrivial.

For an arbitrary homogeneous space $G/H$ we set $\mathcal N(G/H)$ to
be the subset of $G/H$ defined by the vanishing of all functions in
${}^U \mathbb C[G/H]$ that are $T$-semi-invariant and of nonzero
weight.

\begin{proposition} \label{non_empty}
Suppose that $G = G_1 \times \ldots \times G_s$, where each of the
groups $G_i$ is simple. Let $H \subset G$ be a subgroup such that
for every simple component $G_i$ of\, $G$ the projection of $H^0$ to
$G_i$ is not unipotent. Then the set $\mathcal N(G/H) \subset G/H$
is nonempty.
\end{proposition}

\begin{proof}
For $i = 1, \ldots, s$ we put $T_i = T \cap G_i$ so that $T_i$ is a
maximal torus in~$G_i$ and $T = T_1 \times \ldots \times T_s$. We
identify $\mathfrak X(T)$ with a sublattice in $\mathfrak t^*$ by
taking each character $\chi \in \mathfrak X(T)$ to its differential
$d\chi \in \mathfrak t^*$. We regard the rational subspace
$\mathfrak t^*_{\mathbb Q} = \mathfrak X(T) \otimesZ \mathbb Q
\subset \mathfrak t^*$ and fix an inner product $(\cdot\,,\,\cdot)$
on it invariant under the Weyl group $W = N_G(T)/T$. By means of
this inner product we identify $\mathfrak t^*_{\mathbb Q}$ with the
rational subspace
$$
\mathfrak t_{\mathbb Q} = \{x \in \mathfrak t \mid \xi(x) \in
\mathbb Q \text{ for all } \xi \in \mathfrak t^*_{\mathbb Q}\}
\subset \mathfrak t.
$$

Replacing $H$ by a conjugate subgroup, we may assume that the
subgroup $T_H = {(T \cap H)^0} \subset H$ is a maximal torus of~$H$.
We note that $\dim_{\mathbb Q} (\mathfrak t_{\mathbb Q} \cap
\mathfrak t_H) = \dim_{\mathbb C} \mathfrak t_H$. This together with
the hypothesis implies that there exists an element $z \in \mathfrak
t_{\mathbb Q} \cap \mathfrak t_H$ whose projection to each of the
subspaces $\mathfrak t_1, \ldots, \mathfrak t_s$ is nonzero. The
element $z$, regarded as an element of $\mathfrak X(T) \otimesZ
\mathbb Q$, is contained in a Weyl chamber $C \subset \mathfrak X(T)
\otimesZ \mathbb Q$. Again replacing $H$ by a conjugate subgroup
(conjugate by a suitable element of~$N_G(T)$), without loss of
generality we may assume that $C$ is the dominant Weyl chamber.

Suppose that $\lambda \in \Lambda_+(G) \backslash \{0\}$. Regard the
irreducible $G$-module $V(\lambda^*)$, a lowest weight vector
$w_{\lambda^*}$ of it, and a $T$-invariant subspace~$V'(\lambda^*)$
complementary to~$w_{\lambda^*}$. Let $\nu$ be a linear function on
$V(\lambda^*)$ taking $w_{\lambda^*}$ to a nonzero value and
vanishing on~$V'(\lambda^*)$. Then $\nu$ is a highest weight vector
of the irreducible $G$-module $V(\lambda^*)^* \simeq V(\lambda)$.
Let $f \in {}^U \mathbb C[G/H]$ be a $T$-semi-invariant function of
weight~$\lambda$. Under isomorphism~(\ref{relation}), it corresponds
to a vector $v_f \in V(\lambda^*)^H$ such that $f(g) = \nu(g v_f)$
for all $g \in G$. Let us show that $v_f \in V'(\lambda^*)$. For
that, it suffices to check that the vector $w_{\lambda^*}$ is not
$T_H$-invariant. The latter will be fulfilled if we show that $z
w_{\lambda^*} \ne 0$. We have $z w_{\lambda^*} = (z, -
(\lambda^*)^*) w_{\lambda^*} = - (z, \lambda) w_{\lambda^*}$. Recall
the following well-known fact: every two fundamental weights
$\omega$, $\omega'$ of~$G$ satisfy the inequality $(\omega,\omega')
\ge 0$, and the equality is attained if and only if $\omega$,
$\omega'$ are fundamental weights of different simple factors
of~$G$. Since $\lambda \ne 0$, $z \in C$, and the projection of $z$
to each of the subspaces $\mathfrak t_1, \ldots, \mathfrak t_s$ is
nonzero, in view of the above fact, we obtain $(z, \lambda)
> 0$, whence $z w_{\lambda^*} \ne 0$ and $v_f \in V'(\lambda^*)$.
Therefore $f(eH) = 0$. Since $\lambda$ and $f$ are arbitrary, we
obtain $eH \in \mathcal N(G/H)$.
\end{proof}

Let $G/H$ be a homogeneous space such that the algebra $\mathbb
C[G/H]$ is finitely generated. The \textit{null fiber} of the
morphism $\pi_U$ is the subset of $X$ defined by the vanishing of
all the functions in ${}^U \mathbb C[G/H]$ that are
$T$-semi-invariant and of nonzero weight.

\begin{corollary} \label{fiber_non_empty}
In the assumptions of Proposition~\textup{\ref{non_empty}} suppose
that the algebra $\mathbb C[G/H]$ is finitely generated. Then the
null fiber of $\pi_U$ is nonempty and intersects~$G/H$.
\end{corollary}

\section{Symmetric linear actions of tori} \label{section_sym_act}

Suppose we are given a linear action of a quasi-torus $S$ on a
vector space~$V$. We recall that the notation $V \quot S$ stands for
the categorical quotient for the action $S : V$, that is, $V \quot S
= \Spec \mathbb C[V]^S$. We also recall that the null fiber (that
is, the fiber containing zero) of the factorization morphism $V \to
V \quot S$ is said to be the null cone.

For every character $\chi \in \mathfrak X(S)$ we denote by~$V_\chi$
the weight subspace in~$V$ of weight $\chi$ with respect to~$S$. We
put $\Phi = \{\chi \in \mathfrak X(S) \mid V_\chi \ne 0\} \backslash
\{0\} \subset \mathfrak X(S)$. Then, evidently, we have $V = V_0
\oplus \bigoplus \limits_{\chi \in \Phi} V_\chi$.

\begin{dfn}
The linear action $S : V$ is said to be \textit{symmetric} if $\dim
V_\chi = \dim V_{-\chi}$ for every $\chi \in \Phi$.
\end{dfn}

For the rest of this section, we assume that $S$ is a torus and the
action $S : V$ is linear and symmetric.

Put $c = \dim V_0$ and $d = (\dim V - c)/2$. For the action $S : V$
there are (not necessarily different) elements $\chi_1, \ldots,
\chi_d \in \mathfrak X(S)$ such that there is an $S$-module
isomorphism $V \simeq V_0 \oplus \bigoplus \limits_{i = 1}^d
(\mathbb C_{\chi_i} \oplus \mathbb C_{-\chi_i})$, where for $\chi
\in \mathfrak X(S)$ we take $\mathbb C_\chi$ to be the
one-dimensional $S$-module on which $S$ acts by the
character~$\chi$.

\begin{dfn}
The action $S:V$ is said to be \textit{excellent} if the elements
$\chi_1, \ldots, \chi_d$ are linearly independent in~$\mathfrak
X(S)$.
\end{dfn}

The description of linear actions of tori such that the
factorization morphism is equidimensional (see~\cite[\S\,8.1]{VP})
implies that the action $S : V$ is excellent if and only if the
factorization morphism $V \to V \quot S$ is equidimensional.
Further, using the method of supports (see~\cite[\S\,5.4]{VP}), it
is easy to show that the dimension of the null cone of the action $S
: V$ is equal to~$d$. As the dimension of an arbitrary fiber of the
morphism $V \to V \quot S$ does not exceed that of the null cone
(see~\cite[Corollary~1 from Proposition~5.1]{VP}), putting what we
have said above together with the theorem on dimensions of fibers of
a dominant morphism imply the following proposition.

\begin{proposition} \label{dim_quot}
We have $\dim V \quot S \ge c + d$, and the equality is attained if
and only if the action $S:V$ is excellent.
\end{proposition}

\section{The extended weight semigroup of a homogeneous space}\label{section_ews}

Let $H \subset G$ be an arbitrary subgroup. For every character
$\chi\in\mathfrak X(H)$ we regard the subspace
\begin{equation*}
V_{\chi} = \{f \in \mathbb C[G] \mid f(gh) = \chi(h)f(g) \:
\text{for all}\, g \in G, \,h \in H\}
\end{equation*}
of the algebra $\mathbb C[G]$. It is easy to see that the action of
$G$ on the space $\mathbb C[G]$ on the left preserves $V_\chi$ for
every $\chi \in \mathfrak X(H)$. All pairs of the form $(\lambda,
\chi)$, where $\lambda \in \Lambda_+(G)$, $\chi \in \mathfrak X(H)$,
such that $V_{\chi}$ contains the irreducible $G$-submodule with
highest weight~$\lambda$ form a semigroup. This semigroup is said to
be the \textit{extended weight semigroup} of the homogeneous
space~$G/H$. (For a more detailed description
see~\cite[\S\,1.2]{Avd2} or~\cite[\S\,1.2]{AG}.) We denote this
semigroup by $\widehat\Lambda_+(G/H)$. Since $V_0 = \mathbb C[G]^H =
\mathbb C[G/H]$, we obtain
\begin{equation} \label{semigroups}
\Lambda_+(G/H) \simeq \{(\lambda, \chi) \in \widehat\Lambda_+(G/H)
\mid \chi = 0\}.
\end{equation}

We define the subgroup $H_0 \subset H$ to be the common kernel of
all characters of~$H$. This subgroup is normal in~$H$ and
contains~$H'$, therefore the group $H/H_0$ is commutative (and is
thereby a quasi-torus). If $H$ is connected, then $H/H_0$ is also
connected and is thereby a torus. In what follows, we identify the
groups $\mathfrak X(H)$ and $\mathfrak X(H/H_0)$ via the natural
isomorphism $\mathfrak X(H/H_0) \to \mathfrak X(H)$. We have
$\bigoplus \limits_{\chi \in \mathfrak X(H)} V_\chi = \mathbb
C[G]^{H_0} = \mathbb C[G/H_0]$. We note that the action of the group
$T \times H/H_0$ determines a grading on ${}^U \mathbb C[G/H_0]$ by
the semigroup~$\widehat \Lambda_+(G/H)$ ($T$~acts on the left,
$H/H_0$ acts on the right).

We now turn to the situation where $H$ is a spherical subgroup
of~$G$. According to~\cite[Theorem~1]{VK}, the sphericity of~$H$ is
equivalent to the condition that the representation of $G$ on the
space $V_\chi$ is multiplicity free for every $\chi \in \mathfrak
X(H)$. This implies that the action of ${T \times H/H_0}$ on the
space ${}^U \mathbb C[G/H_0]$ is multiplicity free. Further, for a
spherical subgroup $H$ the semigroup $\widehat \Lambda_+(G/H)$ is
free (see~\cite[Theorem~2]{AG}, see also the case of connected $H$
in~\cite[Theorem~1]{Avd2}), hence the algebra ${}^U \mathbb
C[G/H_0]$ is free and is isomorphic to the semigroup algebra of the
semigroup~$\widehat \Lambda_+(G/H)$. Regard the affine space $Y_0 =
\Spec {}^U \mathbb C[G/H_0] \simeq \mathbb C^n$, where $n = \rk
\widehat \Lambda_+(G/H)$. We equip $Y_0$ with a structure of a
vector space in such a way that $(T \times H/H_0)$-semi-invariant
functions that freely generate the algebra ${}^U \mathbb C[G/H_0]$
correspond to the coordinate functions on~$Y_0$. The action of
$H/H_0$ on ${}^U \mathbb C[G/H_0]$ on the right naturally
corresponds to an action of this group on~$Y_0$.

\begin{lemma} \label{action_is_lin}
For a spherical subgroup $H \subset G$ the action $H/H_0 : Y_0$ is
linear. Moreover, if $(\lambda_1, \chi_1)$, $\ldots$, $(\lambda_n,
\chi_n)$ are all the indecomposable elements of the
\textup{(}free\textup{)} semigroup $\widehat \Lambda_+(G/H)$, then
there is an $H/H_0$-module isomorphism $Y_0 \simeq \mathbb
C_{-\chi_1} \oplus \ldots \oplus \mathbb C_{-\chi_n}$.
\end{lemma}

\begin{proof}
Let $f_1, \ldots, f_n \in {}^U \mathbb C[G/H_0]$ be the nonzero $(T
\times H/H_0)$-semi-invariant functions corresponding to the
elements $(\lambda_1, \chi_1)$, $\ldots$, $(\lambda_n, \chi_n)$ of
$\widehat \Lambda_+(G/H)$, respectively. These functions freely
generate the algebra~${}^U \mathbb C[G/H_0]$. Interpreting $f_1,
\ldots, f_n$ as coordinate functions on~$Y_0$, we obtain the
required result.
\end{proof}

\begin{lemma} \label{action_is_sym}
For a reductive spherical subgroup~$H \subset G$ the linear action
${H/H_0 : Y_0}$ is symmetric.
\end{lemma}

\begin{proof}
Using isomorphism~(\ref{relation}), it is not hard to show that for
$\lambda \in \Lambda_+(G)$ and $\chi \in \mathfrak X(H)$ the element
$(\lambda, \chi)$ is contained in $\widehat \Lambda_+(G/H)$ if and
only if the subspace $V(\lambda^*)_\chi^{(H)} \subset V(\lambda^*)$,
consisting of $H$-semi-invariant vectors of weight~$\chi$, is
one-dimensional. Suppose that $(\lambda, \chi) \in \widehat
\Lambda_+(G/H)$. As $H$ is reductive, we have $V(\lambda^*) =
V(\lambda^*)_\chi^{(H)} \oplus W$ for some $H$-invariant subspace $W
\subset V(\lambda^*)$. Let $\xi \in V(\lambda^*)^* \simeq
V(\lambda)$ be the linear function on~$V(\lambda^*)$ taking a basis
vector of $V(\lambda^*)_\chi^{(H)}$ to~$1$ and vanishing on~$W$.
Then $\xi$ is an $H$-semi-invariant element in~$V(\lambda)$ of
weight~$-\chi$. Therefore $(\lambda^*, -\chi) \in \widehat
\Lambda_+(G/H)$. Thus the map $(\lambda, \chi) \mapsto (\lambda^*,
-\chi)$ is an automorphism of the semigroup~$\widehat
\Lambda_+(G/H)$. In particular, under this automorphism
indecomposable elements are taken into indecomposable elements.
Hence in view of Lemma~\ref{action_is_lin} we obtain the required
result.
\end{proof}

\section{Reduction of the proof of Theorem~\ref{theorem_reformulation}\\
to the case of strictly irreducible spaces}\label{section_reduction}

Suppose that $G$ is simply connected and $H \subset G$ is a
connected reductive spherical subgroup. We recall that in
\S\,\ref{section_ews} we took $H_0$ to be the common kernel of all
characters of~$H$ and introduced the notation $Y_0$ for the affine
space $\Spec {}^U \mathbb C[G/H_0] \simeq \mathbb C^n$, where $n =
\rk \widehat \Lambda_+(G/H)$. We put $X_0 = \Spec \mathbb C[G/H_0]$.
In our situation, we have $X = G/H$, $X_0 = G/H_0$.

The commutative diagram
\begin{equation}
\xymatrix{
\mathbb C[G/H_0] & {}^U \mathbb C[G/H_0] \ar[l]\\
\mathbb C[G/H] \ar[u] & {}^U \mathbb C[G/H] \ar[l] \ar[ul] \ar[u] }
\end{equation}
of injective homomorphisms of algebras corresponds to the
commutative diagram
\begin{equation}
\xymatrix{
G/H_0 = X_0\hspace{-3.78em}& {\vphantom{/}}{\phantom{X_0}} \ar[r]^{\varphi_U} \ar[d]_{\psi_X} \ar[rd]^{\psi} & {\vphantom{/}}Y_0 \ar[d]^{\psi_Y}\\
G/H = X\hspace{-4.05em} & {\vphantom{/}}{\phantom{X}} \ar[r]_{\pi_U}
& {\vphantom{/}}Y }
\end{equation}
of dominant morphisms of the respective varieties. We recall that
the affine space $Y_0$ is equipped with the structure of a vector
space (see~\S\,\ref{section_ews}) and that the action $H/H_0 : Y_0$
is linear (Lemma~\ref{action_is_lin}) and symmetric
(Lemma~\ref{action_is_sym}). In this situation, $\psi_Y$ is nothing
else but the factorization morphism for the action $H/H_0 : Y_0$.

\begin{proposition} \label{crucial}
If the morphism $\pi_U$ is equidimensional, then the action ${H/H_0
: Y_0}$ is excellent.
\end{proposition}

\begin{proof}
Suppose that $\pi_U$ is equidimensional. Since the morphism $\psi_X$
is also equidimensional, it follows that the morphism $\psi \colon
X_0 \to Y$ is equidimensional. Assume that the action $H/H_0 : Y_0$
is not excellent. Let $c$ and $d$ be the corresponding
characteristics of this action (see \S\,\ref{section_sym_act}). In
view of Proposition~\ref{dim_quot} we have $\dim Y > c + d$, whence
by the theorem on dimensions of fibers of a dominant morphism the
codimension of a generic fiber of $\psi$ is greater than $c + d$.

It follows from Theorem~\ref{classification} that there are simply
connected strictly irreducible affine spherical homogeneous spaces
$G_1/H_1$, $\ldots$, $G_n/H_n$ such that $G = G_1 \times \ldots
\times G_n$ and $H'_1 \times \ldots \times H'_n \subset H \subset
H_1 \times \ldots \times H_n$. Then $G/H_0 \simeq G_1/H'_1 \times
\ldots \times G_n/H'_n$. For every $i = 1, \ldots, n$ we put $U_i =
U \cap G_i$, $T_i = T \cap G_i$. In each algebra ${}^{U_i} \mathbb
C[G_i/H'_i] \subset {}^U \mathbb C[G/H_0]$, $i = 1, \ldots, n$, we
fix a subset $F_i$, consisting of $(T_i \times
H_i/H'_i)$-semi-invariant functions that freely generate this
algebra. Then $F = F_1 \cup \ldots \cup F_n$ is a set of $(T \times
H/H_0)$-semi-invariant functions that freely generate the
algebra~${}^U \mathbb C[G/H_0]$.

Let $\Phi \subset \mathfrak X(H/H_0)$ be the set of nonzero
$H/H_0$-weights of all functions in~$F$. In the space $\mathfrak
X(H/H_0) \otimesZ \mathbb Q$, fix a hyperplane $h$ that contains no
elements of~$\Phi$. Choose one of the half-spaces bounded by~$h$ and
denote by $\Phi^+$ the set of weights in $\Phi$ that are contained
in this half-space. We put $\Phi^- = -\Phi^+$ so that $\Phi = \Phi^+
\cup \Phi^-$. Let $F^+$ (resp.~$F^0$,~$F^-$) be the set of functions
in~$F$ whose $H/H_0$-weights belong to $\Phi^+$
(resp.~$\{0\}$,~$\Phi^-$). For every $i = 1, \ldots, n$ put also
$F^+_i = F^+ \cap F_i$, $F^0_i = F^0 \cap F_i$, $F^-_i = F^- \cap
F_i$. Clearly, $c = |F^+| = |F^-|$ and $d = |F^0|$. It is not hard
to deduce (see~\cite[\S\,5.4]{VP}) that the subset of~$Y_0$ defined
by the vanishing of all the $c + d$ functions in $F^+ \cup F^0$ is
contained in the null fiber of~$\psi_Y$. Then the subset $\mathcal
N$ of $X_0$ defined by the vanishing of the same functions is
contained in the fiber $\psi^{-1}(\psi_Y(0))$ of the morphism~$\psi$
(below we refer to this fiber as the null fiber as well). Let us
show that $\mathcal N \ne \varnothing$. As can be easily seen, it is
enough to show that the subset of $G_i/H'_i$ defined by the
vanishing of all the functions in $F^+_i \cup F^0_i$ is nonempty for
every $i = 1, \ldots, n$. To do that, we consider two possibilities.

(1) The group $H'_i$ is trivial. Then $H_i$ is a torus. Inspecting
the list of all simply connected strictly irreducible affine
spherical homogeneous spaces we find that this is only possible for
$G_i = \SL_2$, ${H_i \simeq \mathbb C^\times}$. In this case, $F_i$
contains two functions having nonzero opposite weights with respect
to the torus~$H_i/H'_i$. As $\SL_2/\mathbb C^\times$ is a strictly
irreducible spherical homogeneous space of type~I, by
Theorem~\ref{classification}(a) the images of these weights in
$\mathfrak X(H/H_0)$ are also nonzero. Hence $|F^+_i| = |F^-_i| = 1$
and $F_i = F^+_i \cup F^-_i$. A direct check shows that the subset
of $\SL_2$ defined by the vanishing of the unique function in
$F^+_i$ is nonempty.

(2) The group $H'_i$ is nontrivial. Inspecting the list of all
simply connected strictly irreducible affine spherical homogeneous
spaces we find that in this case the group $H'_i$ satisfies the
hypothesis of Proposition~\ref{non_empty}, hence even the subset of
$G_i/H'_i$ defined by the vanishing of all functions in $F_i$ is
nonempty.

Thus, $\mathcal N \ne \varnothing$. Since $\mathcal N$ is contained
in the null fiber of $\psi$ and is defined in~$X_0$ by the vanishing
of $c+d$ functions, the codimension in $X_0$ of the null fiber of
$\psi$ is at most $c + d$ and therefore is strictly less than the
codimension of a generic fiber. Therefore the morphism $\psi$ is not
equidimensional, a contradiction.
\end{proof}

\begin{proposition} \label{ed_product}
Suppose that $G/H$ is a simply connected affine spherical
homogeneous space such that the morphism $\pi_U$ is equidimensional.
Then $G/H$ is a direct product of several \textup{(}simply
connected\textup{)} strictly irreducible affine spherical
homogeneous spaces.
\end{proposition}

\begin{proof}
By Theorem~\ref{classification} there are simply connected strictly
irreducible affine spherical homogeneous spaces $G_1/H_1$, $\ldots$,
$G_n/H_n$ such that:

(1) $G = G_1 \times \ldots \times G_n$;

(2) $H'_1 \times \ldots \times H'_n \subset H \subset H_1 \times
\ldots \times H_n$.

Without loss of generality we may assume that the space $G_i/H_i$ is
of type~I or~II for $i \le q$ and is of type~III for $i > q$. For
every $i = 1, \ldots, n$, put $U_i = G_i \cap U$.

Let $\tau \colon \mathfrak X(H_1 \times \ldots \times H_n) \to
\mathfrak X(H)$ be the character restriction map. For $i = 1,
\ldots, q$ we denote by $\chi_i$ a basis character in $\mathfrak
X(H_i) = \mathfrak X(H_i/H'_i)$. Renumbering if necessary, we may
assume that $\tau(\chi_i) \ne 0$ for $i \le p$ and $\tau(\chi_i) =
0$ for $p < i \le q$. Since for every $i = 1, \ldots, p$ the
(linear) action of the torus $H_i/H'_i$ on the affine space $\Spec
{}^{U_i} \mathbb C[G_i/H'_i]$ is nontrivial,
Proposition~\ref{crucial} implies that the weights $\tau(\chi_1)$,
$\ldots$, $\tau(\chi_p)$ are pairwise different and linearly
independent in~$\mathfrak X(H)$. It follows that $H = H_1 \times
\ldots \times H_p \times H'_{p+1} \times \ldots \times H'_n$ and
$G/H \simeq G_1/H_1 \times \ldots \times G_p/H_p \times
G_{p+1}/H'_{p+1} \times \ldots \times G_n/H'_n$.
\end{proof}

\section{The case of strictly irreducible spaces} \label{section_str_irr}

In view of Proposition~\ref{ed_product}, to complete the proof of
Theorem~\ref{theorem_reformulation} it remains to prove the
following proposition.

\begin{proposition} \label{prop_NE_NED}
Suppose that $G/H$ is a simply connected strictly irreducible affine
spherical homogeneous space that is not excellent. Then the morphism
$\pi_U$ is not equidimensional.
\end{proposition}

\begin{proof}
Since we have at our disposal the classification (or more precisely,
a complete list) of simply connected strictly irreducible affine
spherical homogeneous spaces (see \S\,\ref{section_classification}),
it suffices to consider all nonexcellent spaces among them case by
case and prove the assertion by a direct check in each case. We
recall that, up to an isomorphism, the list of all simply connected
strictly irreducible affine spherical homogeneous spaces is
collected in Tables~1 and~2 of the paper~\cite{Avd1}. In these
tables it is also indicated whether each of the spaces is excellent
or not. Hence we obtain a list of spaces we need to consider in
order to prove the proposition. Before we proceed to this list, let
us introduce some additional conventions and notation.

If $G$ is a product of several factors, then we denote by $\pi_i$,
$\varphi_i$, and $\psi_i$ the $i$th fundamental weight of the first,
second, and third factor, respectively, and we use the same
numeration of the fundamental weights of simple groups as in the
book~\cite{VO}.

We denote by $E_m$ the identity matrix of order~$m$ and by $F_m$ the
matrix of order $m$ with ones on the antidiagonal and zeros
elsewhere.

For every matrix denoted by a capital letter, the corresponding
small letter with a double index $ij$ stands for the element in the
$i$th row and the $j$th column of this matrix. For example, $P$ is a
matrix and $p_{ij}$ is the element in its $i$th row and $j$th column
of~$P$.

A basis $e_1, \ldots, e_n$ of the space of the tautological linear
representation of the group $\SO_n$ is assumed to be chosen in such
a way that the matrix of the invariant nondegenerate symmetric
bilinear form is~$F_n$. A basis $e_1, \ldots, e_{2m}$ of the space
of the tautological linear representation of the group $\Sp_{2m}$ is
assumed to be chosen in such a way that the invariant nondegenerate
skew-symmetric bilinear form has the matrix
$$
\begin{pmatrix}
0 & F_m\\
-F_m & 0
\end{pmatrix}.
$$
With these choices of bases we may (and shall) assume that for every
simple factor $\overline G$ of $G$ the groups $B \cap \overline G$,
$U \cap \overline G$, and $T \cap \overline G$ consist of the
upper-triangular, upper unitriangular, and diagonal matrices
contained in $\overline G$.

We now proceed to a consideration of all the simply connected
strictly irreducible affine spherical homogeneous spaces $G / H$
that are not excellent. We divide the consideration into two cases
depending on the value~$\dim Z(H)$.

\textit{Case}~1. $\dim Z(H) = 1$. For every space $G/H$ under
consideration we indicate the indecomposable elements of the
semigroup~$\widehat \Lambda_+(G/H)$. By Lemma~\ref{action_is_lin}
these elements completely determine the $H/H_0$-module structure
in~$Y_0$. For every space $G/H$ it is easy to check that the action
$H/H_0 : Y_0$ is not excellent, whence by Proposition~\ref{crucial}
the morphism $\pi_U$ is not equidimensional.

$1^\circ$. $G = \SL_{2n+1}$, $H = \mathbb C^\times \times \Sp_{2n}$,
$n \ge 2$. The factor $\Sp_{2n}$ of $H$ is embedded in $G$ as the
upper left $2n \times 2n$ block, and the torus $\mathbb C^\times$ is
embedded in $G$ via the map $s \mapsto \diag(s, \ldots, s,
s^{-2n})$.

The information contained in rows~6 and~7 of Table~1 in~\cite{Avd1}
enables us to conclude that the semigroup $\widehat \Lambda_+(G/H)$
is freely generated by the elements $(\pi_1, n\chi)$, $(\pi_2,
-\chi)$, ${(\pi_3, (n-1)\chi)}$, $(\pi_4, -2 \chi)$, $\ldots$,
$(\pi_{2n-1}, \chi)$, $(\pi_{2n}, -n\chi)$ for some nonzero
character~${\chi \in \mathfrak X(H)}$. As $n \ge 2$, it follows that
the action $H/H_0 : Y_0$ is not excellent and the morphism $\pi_U$
is not equidimensional.

$2^\circ$. $G = \Spin_{10}$, $H = \mathbb C^\times \times \Spin_7$.
This homogeneous space is uniquely determined by the homogeneous
space $\widetilde G/\widetilde H$, locally isomorphic to~$G/H$,
where $\widetilde G = \SO_{10}$, $\widetilde H = \mathbb C^\times
\times \Spin_7$. The factor $\Spin_7$ of $\widetilde H$ is regarded
as a subgroup of~$\SO_8$ (the embedding $\Spin_7 \hookrightarrow
\SO_8$ is given by the spinor representation), and the group $\SO_8$
is embedded in $\widetilde G$ as the central $8 \times 8$ block. The
torus $\mathbb C^\times$ is embedded in $\widetilde G$ via the map
$s \mapsto \diag(s,1,1, \ldots, 1, s^{-1})$.

The information contained in row~16 of Table~1 in~\cite{Avd1} and in
row~13 of Table~1 in~\cite{Pan1} enables us to conclude that the
semigroup $\widehat \Lambda_+(G/H)$ is freely generated by the
elements $(\pi_1, 2\chi)$, $(\pi_1, -2\chi)$, $(\pi_2, 0)$, $(\pi_4,
\chi)$, $(\pi_5, -\chi)$ for some nonzero character~${\chi \in
\mathfrak X(H)}$. It follows that the action $H/H_0 : Y_0$ is not
excellent and the morphism $\pi_U$ is not equidimensional.

$3^\circ$. $G = \SL_n \times \SL_{n+1}$, $H = \SL_n \times \mathbb
C^\times$, $n \ge 2$. The factor $\SL_n$ of $H$ is diagonally
embedded in $G$, as the upper left $n \times n$ block in the factor
$\SL_{n+1}$. The torus $\mathbb C^\times$ is embedded in the factor
$\SL_{n+1}$ of $G$ via the map $s \mapsto \diag(s, \ldots, s,
s^{-n})$.

It is indicated in row~1 of Table~1 in~\cite{Avd2} that the
semigroup $\widehat \Lambda_+(G/H)$ is freely generated by the
elements $(\varphi_1, n\chi)$, $(\pi_{n-1} + \varphi_2, (n-1)\chi)$,
$\dots$, $(\pi_1 + \varphi_n, \chi)$, $(\pi_{n-1} + \varphi_1,
-\chi)$, $\dots$, $(\pi_1 + \varphi_{n-1}, -(n-1)\chi)$,
$(\varphi_n, -n\chi)$ for some nonzero character~$\chi \in \mathfrak
X(H)$. It follows that the action $H/H_0 : Y_0$ is not excellent and
the morphism $\pi_U$ is not equidimensional.

$4^\circ$. $G = \SL_n \times \Sp_{2m}$, $H = \mathbb C^\times \cdot
\SL_{n-2} \times \SL_2 \times \Sp_{2m-2}$, $n \ge 3$, $m \ge 1$. The
embedding of $H$ in $G$ is as follows. The factor $\SL_{n-2}$ is
embedded in the factor $\SL_n$ of $G$ as the upper left $(n-2)\times
(n-2)$ block. The factor $\SL_2$ is diagonally embedded in~$G$ as
the lower right $2\times 2$ block in the factor $\SL_n$ and as the
$2\times 2$ block corresponding to the first and the last rows and
columns in the factor~$\Sp_{2m}$. The factor $\Sp_{2m-2}$ of $H$ is
embedded in the factor $\Sp_{2m}$ of $G$ as the central
$(2m-2)\times (2m-2)$ block. At last, the torus $\mathbb C^\times$
embedded in the factor $\SL_n$ of $G$ via the map $s \mapsto
\diag(s^{-2}, \ldots, s^{-2}, s^{n-2}, s^{n-2})$ for odd $n$ and via
the map $s \mapsto \diag(s^{-1}, \ldots, s^{-1}, s^{\frac{n-2}2},
s^{\frac{n-2}2})$ for even~$n$.

It is indicated in row~3 of Table~1 in~\cite{Avd2} that the
semigroup $\widehat \Lambda_+(G/H)$ is freely generated by the
elements $(\pi_{n-2}, 2\chi)$, $(\varphi_2, 0)$ (this element is
contained in the set of indecomposable elements for~${m \ge 2}$),
${(\pi_{n-1} + \varphi_1, \chi)}$, $(\pi_1 + \pi_{n-1}, 0)$ (this
element is contained in the set of indecomposable elements for~$n
\ge 4$), ${(\pi_1 + \varphi_1, -\chi)}$, $(\pi_2, -2\chi)$ for some
nonzero character~$\chi \in \mathfrak X(H)$. It follows that the
action $H/H_0 : Y_0$ is not excellent and the morphism $\pi_U$ is
not equidimensional.

\textit{Case}~2. $\dim Z(H) = 0$. For each of the spaces $G/H$ below
we fix and denote by $F$ a set of $T$-semi-invariant functions that
freely generate the algebra ${}^U \mathbb C[G/H]$. The weights of
these functions are the indecomposable elements of the (free)
semigroup $\Lambda_+(G/H) \simeq \widehat \Lambda_+(G/H)$. We put
$|F| = r$. Note that $r$ is the codimension of a generic fiber of
the morphism~$\pi_U$. In all the cases the null fiber of $\pi_U$ is
nonempty by Proposition~\ref{non_empty}. To consider the spaces of
series $2^\circ$ and $4^\circ$ we shall need the auxiliary lemma
below, which follows from the theorem on dimensions of fibers of a
dominant morphism.

\begin{lemma} \label{auxiliary}
Let $\gamma \colon M \to N$ be a morphism of affine algebraic
varieties. Suppose that a closed subvariety $N_0 \subset N$ is such
that the set ${\gamma(M) \cap N_0}$ is dense in~$N_0$. Then
$\codim_M \gamma^{-1}(N_0) \le \codim_N N_0$.
\end{lemma}

Below we treat all the required spaces $G/H$.

$1^\circ$. $G = \Spin_n \times \Spin_{n+1}$, $H = \Spin_n$, $n \ge
3$. The homogeneous space $G/H$ is uniquely determined by the
homogeneous space $\widetilde G / \widetilde H$, locally isomorphic
to $G/H$, where $\widetilde G = \SO_n \times \SO_{n+1}$, $\widetilde
H = \SO_n$, and the subgroup $\widetilde H$ is diagonally embedded
in~$\widetilde G$. Let $\theta_m \colon \SO_m \hookrightarrow
\SO_{m+1}$ be the embedding induced by the embedding $\mathbb C^m
\hookrightarrow \mathbb C^{m+1}$ sending the basis $e_1, \ldots,
e_m$ to the tuple $e_1, \ldots, e_{\frac{m}2}, e_{\frac{m}2+2},
\ldots e_{m+1}$ for even $m$ and to the tuple $e_1, \ldots,
e_{\frac{m-1}2}, {\frac{1}{\sqrt{2}}(e_{\frac{m+1}2} +
e_{\frac{m+3}2})}, e_{\frac{m+5}2}, \ldots, e_{m+1}$ for odd~$m$.
The image of $\theta_m$ is the stabilizer of the vector
$e_{\frac{m}2 + 1}$ for even~$m$ and the vector $e_{\frac{m+1}2} -
e_{\frac{m+3}2}$ for odd~$m$. We fix the embedding of $\widetilde H$
in $\widetilde G$ such that the image in $\widetilde G$ of the
matrix ${P \in \widetilde H}$ is $(P,\theta_n(P))$. The covering
$G/H \to \widetilde G / \widetilde H$ determines the natural
embedding $\mathbb C[\widetilde G / \widetilde H] \hookrightarrow
\mathbb C[G/H]$. In view of this embedding every regular function on
$\widetilde G / \widetilde H$ will also be considered as a regular
function on~$G/H$.

For $n = 3$ there is an isomorphism between the homogeneous space
$G/H$ and the space $(\SL_2 \times \SL_2 \times \SL_2) / \SL_2$,
where the subgroup $\SL_2$ is diagonally embedded in $\SL_2 \times
\SL_2 \times \SL_2$. This space is isomorphic to the space of
series~$4^\circ$ (see below) with $n = m = l = 1$. For $n = 4$ the
homogeneous space $G/H$ is isomorphic to the space of
series~$5^\circ$ (see below) with $n = m = 1$. Therefore, below we
assume that~$n \ge 5$.

Suppose that $g = (P,Q) \in \widetilde G$. Put $R =
Q\theta_n(P)^{-1}$. For $n = 2k$ the set $F$ contains functions that
are proportional to the following functions of~$g$: $f_1 =
r_{n+1,1}$, $f_2 = r_{n+1,k+1}$, and $f_3 = f_1r_{n,k+1} -
f_2r_{n,1}$; for $n = 2k + 1$ the set $F$ contains functions that
are proportional to the following functions of~$g$: $f_1 =
r_{n+1,1}$, $f_2 = r_{n+1,k+1} - r_{n+1,k+2}$, and $f_3 =
f_1(r_{n,k+1} - r_{n,k+2}) - f_2r_{n,1}$ (see~\cite[\S\,3.1,
Case~2]{Avd2}). In what follows, without loss of generality we
assume that $f_1, f_2, f_3 \in F$. The $T$-weights of the functions
$f_1, f_2$ are $\pi_1 + \varphi_1$, $\varphi_1$, respectively, and
the $T$-weight of $f_3$ is $\pi_1 + \varphi_2$ for $n \ge 6$ and
$\pi_1 + \varphi_2 + \varphi_3$ for $n = 5$. Since the condition
$f_1 = f_2 = 0$ implies $f_3 = 0$, the subset of $G/H$ defined by
the vanishing of all functions in $F$ coincides with the subset of
$G/H$ defined by the vanishing of all functions in $F \backslash
\{f_3\}$. Thus the codimension of the null fiber of the morphism
$\pi_U$ is at most $r - 1$, hence $\pi_U$ is not equidimensional.

$2^\circ$. $G = \SL_n \times \Sp_{2m}$, $H = \SL_{n-2} \times \SL_2
\times \Sp_{2m-2}$, $n \ge 5$, $m \ge 1$. The factor $\SL_{n-2}$ of
$H$ is embedded in the factor $\SL_n$ of $G$ as the upper left
$(n-2)\times (n-2)$ block. The factor $\SL_2$ of $H$ is diagonally
embedded in~$G$ as the lower right $2\times 2$ block in $\SL_n$ and
as the $2\times 2$ block corresponding to the first and the last
rows and columns in the factor~$\Sp_{2m}$. The factor $\Sp_{2m-2}$
of $H$ is embedded in the factor $\Sp_{2m}$ of $G$ as the central
$(2m-2)\times (2m-2)$ block.

Suppose that $g = (P,Q) \in G$. We denote by $P_{ij}$ the
$(i,j)$-cofactor of the matrix~$P$. The set $F$ contains functions
that are proportional to the following functions of~$g$: $f_1 =
p_{n,n-1} q_{2m,2m} - p_{n,n} q_{2m,1}$, $f_2 = p_{n,n-1} P_{1,n-1}
+ p_{nn} P_{1,n}$, $f_3 = q_{2m,1} P_{1,n-1} + q_{2m,2m} P_{1,n}$
(see~\cite[\S\,3.2, Case~4]{Avd2}). The $T$-weights of the functions
$f_1, f_2, f_3$ are $\pi_{n-1} + \varphi_1$, $\pi_1 + \pi_{n-1}$,
$\pi_1 + \varphi_1$, respectively. Let $\gamma \colon G/H \to
\mathbb C^6$ be the morphism defined by the functions $P_{1,n-1}$,
$P_{1,n}$, $p_{n,n-1}$, $p_{n,n}$, $q_{2m,1}$, $q_{2m,2m}$. Let
$N_0$ denote the subset of~$\mathbb C^6$ defined by~$f_1 = f_2 = f_3
= 0$. As can be easily seen,~$\dim N_0 = 4$. Now, given nonzero
numbers $a,b,c,d$ we consider the pair of matrices
\begin{equation*}
P_0 =
\begin{pmatrix}
0 & \dots & 0 & b^{-1}\\
\vdots & E_{n-3} & \vdots & \vdots\\
d^{-1} & \dots & 0 & 0\\
0 & \dots & bd & -ad\\
\end{pmatrix},\quad
Q_0 =
\begin{pmatrix}
a^{-1}c^{-1} & \dots & 0\\
\vdots & E_{2m-2} & \vdots\\
-bc & \dots & ac\\
\end{pmatrix}
\end{equation*}
(the dots stand for zero entries). We have $(P_0, Q_0) \in G$ for
every nonzero values of~$a,b,c,d$, and the values of the functions
$P_{1,n-1}$, $P_{1,n}$, $p_{n,n-1}$, $p_{n,n}$, $q_{2m,1}$,
$q_{2m,2m}$ on this pair are $a$, $b$, $bd$, $-ad$, $-bc$, $ac$,
respectively. Besides, it is easy to check that $\gamma ((P_0,Q_0)H)
\in N_0$. It follows that $\dim (\gamma(G/H) \cap N_0) = 4$, that
is, the set ${\gamma(G/H) \cap N_0}$ is dense in~$N_0$. Then by
Lemma~\ref{auxiliary} the codimension in $G/H$ of the subset defined
by the vanishing of the functions $f_1, f_2, f_3$ is at most~$2$.
Hence the codimension in $G/H$ of the subset defined by the
vanishing of all functions in $F$ is at most $r - 1$. Therefore the
morphism $\pi_U$ is not equidimensional.

$3^\circ$. $G = \Sp_{2n} \times \Sp_4$, $H = \Sp_{2n-4} \times
\Sp_4$, $n \ge 3$. The first factor of $H$ is embedded in the first
factor of $G$ as the central $(2n-4)\times(2n-4)$ block, and the
second factor of $H$ is diagonally embedded in~$G$, as the $4\times
4$ block in rows and columns $1$, $2$, $2n-1$, $2n$ in the first
factor.

Suppose that $g = (P,Q) \in G$. We put $R = PQ^{-1} \in \Sp_{2n}$
(the matrix $Q$ is embedded in $\Sp_{2n}$ as the $4\times 4$ block
in rows and columns $1$, $2$, $2n-1$, $2n$). Let $W$ denote the
minor of order~3 of~$R$ corresponding to the last three rows and
columns $1$, $2$, $2n$.  The set $F$ contains functions that are
proportional to the following functions of~$g$: $f_1 = r_{2n,1}$,
$f_2$~ is the minor of order~3 of~$R$ corresponding to the last
three rows and columns $1$, $2$, $2n-1$, and $f_3 = f_1 W +
f_2r_{2n,2}$ (see~\cite[\S\,3.2, Case~6]{Avd2}). Further without
loss of generality we assume that $f_1, f_2, f_3 \in F$. The
$T$-weights of the functions $f_1, f_2, f_3$ are $\pi_1 +
\varphi_1$, $\pi_3 + \varphi_1$, $\pi_1 + \pi_3 + \varphi_2$,
respectively. As the condition $f_1 = f_2 = 0$ implies $f_3 = 0$,
the subset of $G/H$ defined by the vanishing of all functions in $F$
coincides with the subset of $G/H$ defined by the vanishing of all
functions in $F \backslash \{f_3\}$. Thus the codimension of the
null fiber of the morphism $\pi_U$ is at most $r - 1$, hence $\pi_U$
is not equidimensional.

$4^\circ$. $G = \Sp_{2n} \times \Sp_{2m} \times \Sp_{2l}$, $H =
\Sp_{2n-2} \times \Sp_{2m-2} \times \Sp_{2l-2} \times \Sp_2$, $n,m,l
\ge 1$. Each of the first three factors of $H$ is embedded in the
respective factor of $G$ as the central block of the corresponding
size. The factor $\Sp_2$ of $H$ is diagonally embedded in $G$ as the
$2\times 2$ block corresponding to the first and the last rows and
columns in each factor.

Suppose that $g = (P,Q,R) \in G$. The set $F$ contains functions
proportional to the following functions of~$g$: $f_1 =
p_{2n,1}q_{2m,2m} - p_{2n,2n}q_{2m,1}$, $f_2 = q_{2m,1}r_{2l,2l} -
q_{2m,2m}r_{2l,1}$, and~$f_3 = p_{2n,1}r_{2l,2l} -
p_{2n,2n}r_{2l,1}$ (see~\cite[\S\,3.2, Case~7]{Avd2}). The
$T$-weights of the functions $f_1, f_2, f_3$ are $\pi_1 +
\varphi_1$, $\varphi_1 + \psi_1$, $\pi_1 + \psi_1$, respectively.
Let $\gamma \colon G/H \to \mathbb C^6$ be the morphism defined by
the functions $p_{2n,1}$, $p_{2n,2n}$, $q_{2m,1}$, $q_{2m,2m}$,
$r_{2l,1}$, $r_{2l,2l}$. Let $N_0$ denote the subset in~$\mathbb
C^6$ defined by~$f_1 = f_2 = f_3 = 0$. As can be easily seen,~$\dim
N_0 = 4$. Each of the functions $p_{2n,1}$, $p_{2n,2n}$, $q_{2m,1}$,
$q_{2m,2m}$, $r_{2l,1}$, $r_{2l,2l}$ is semi-invariant with respect
to the action of the group $T \times T$, where the left factor acts
on the left and the right factor acts on the right. The weights of
all these functions are linearly independent in~$\mathfrak X(T
\times T)$. It follows that for every quadruple of nonzero numbers
$a,b,c,d$ there is a triple of matrices $(P_0, Q_0, R_0) \in G$ such
that the values of the functions $p_{2n,1}$, $p_{2n,2n}$,
$q_{2m,1}$, $q_{2m,2m}$, $r_{2l,1}$, $r_{2l,2l}$ on $(P_0, Q_0,
R_0)$ are $a$, $ad$, $b$, $bd$, $c$, $cd$, respectively. Besides, it
is easy to check that $\gamma ((P_0, Q_0, R_0)H) \in N_0$ for any
nonzero values of $a,b,c,d$. This implies that $\dim (\gamma(G/H)
\cap N_0) = 4$, that is, the set ${\gamma(G/H) \cap N_0}$ is dense
in~$N_0$. Then by Lemma~\ref{auxiliary} the codimension in $G/H$ of
the subset defined by the vanishing of the functions $f_1, f_2, f_3$
is at most~$2$. Therefore the codimension in $G/H$ of the set
defined by the vanishing of all functions in $F$ is at most $r - 1$.
Hence the morphism $\pi_U$ is not equidimensional.

$5^\circ$. $G = \Sp_{2n} \times \Sp_4 \times \Sp_{2m}$, $H =
\Sp_{2n-2} \times \Sp_2 \times \Sp_2 \times \Sp_{2m-2}$, $n, m \ge
1$. The first factor of $H$ is embedded in the first factor of $G$
as the central $(2n - 2) \times (2n - 2)$ block. The fourth factor
of $H$ is similarly embedded in the third factor of~$G$. The second
factor of $H$ is diagonally embedded in the first and second factors
of $G$ as the $2\times 2$ block in the first and the last rows and
columns. The third factor of $H$ is diagonally embedded in the
second and third factors of~$G$ as the central $2\times 2$ block in
the second factor and as the $2\times 2$ block in the first and the
last rows and columns in the third factor.

Suppose that $g = (P,Q,R) \in G$. The set $F$ contains functions
that are proportional to the following functions of~$g$: $f_1 =
p_{2n,1}q_{44} - p_{2n,2n}q_{41}$, $f_2 = r_{2m,1}q_{43} -
r_{2m,2m}q_{42}$, and~$f_3 = f_2(p_{2n,1}q_{34} - p_{2n,2n}q_{31}) -
f_1(r_{2m,1}q_{33} - r_{2m,2m}q_{32})$ (see~\cite[\S\,3.2,
Case~8]{Avd2}). Further without loss of generality we assume that
$f_1, f_2, f_3 \in F$. The $T$-weights of the functions $f_1, f_2,
f_3$ are $\pi_1 + \varphi_1$, $\varphi_1 + \psi_1$, $\pi_1 +
\varphi_2 + \psi_1$, respectively. As the condition $f_1 = f_2 = 0$
implies that~$f_3 = 0$, the subset of $G/H$ defined by the vanishing
of all functions in $F$ coincides with the subset of $G/H$ defined
by the vanishing of all functions in $F \backslash \{f_3\}$. Thus
the codimension of the null fiber of the morphism $\pi_U$ is at most
$r - 1$, therefore $\pi_U$ is not equidimensional.

The proof of Proposition~\ref{prop_NE_NED} is completed.
\end{proof}

\end{document}